\documentclass[11pt,reqno]{amsart}      
\usepackage{cite}                       
\usepackage{amssymb,amsthm}             

\usepackage{paralist}                   
\usepackage{calc}                       
\usepackage{graphicx}
\usepackage{latexsym}                   

\numberwithin{figure}{section}          

\numberwithin{equation}{section}        
\renewcommand{\Re}{{\mathbb R}}         
\newcommand{\half}{\frac{1}{2}}         
\theoremstyle{plain}
\newtheorem{thm}{Theorem}[section]

\newtheorem{lemma}[thm]{Lemma}
\newtheorem{definition}[thm]{Definition}

\title[Extended Applicability of the Symplectic Pontryagin Method]{Extended Applicability of the Symplectic Pontryagin Method}
\author{Mattias Sandberg}
\address{Centre of Mathematics for Applications (CMA) c/o Dept of
  Mathematics\\
Box 1053 Blindern\\
NO-0316 Oslo\\
NORWAY
}
\email{mattias.sandberg@cma.uio.no}

\begin{document}
\begin{abstract}
The Symplectic Pontryagin method was introduced in a previous
paper. This work shows that this method is applicable under less restrictive
assumptions. Existence of solutions to the Symplectic Pontryagin
scheme are shown to exist without the previous assumption on a bounded
gradient of the discrete dual variable. The convergence proof uses the
representation of solutions to a Hamilton-Jacobi-Bellman equation as
the value function of an associated variation problem. 
\end{abstract}
\maketitle
\tableofcontents
\section{Introduction}
The Symplectic Pontryagin method was introduced in
\cite{Sandberg-Szepessy}, as a numerical method designed for
approximation  of the value function associated with   optimal control
problems. Under general conditions this value function, $u:\Re\sp
d\times[0,T]\rightarrow \Re$, is a viscosity
solution of an associated Hamilton-Jacobi equation,
\begin{equation}\label{eq:HJ}
\begin{split}
u_t+H(x,u_x)&=0,\quad\text{in }\Re\sp d\times(0,T), \\
u(x,T)&=g(x),
\end{split}
\end{equation}
where $u_t$ denotes the derivative with respect to the ``time'' variable $t$, and $u_x$ the
gradient with respect to the state variable $x$. 
The function $H:\Re\sp d\times\Re\sp d\rightarrow\Re$ is denoted the
\emph{Hamiltonian}. This function is concave in its second argument
when the Hamilton-Jacobi equation originates in optimal control. We
shall therefore only deal with such concave Hamiltonians in this paper.
The main ingredient
in the Symplectic Pontryagin method is the 
observation that, under favorable conditions,
optimal paths to the control problem solve a Hamiltonian system
associated with the Hamilton-Jacobi equation \eqref{eq:HJ}. More
accurately, if $x:[s,T]\rightarrow\Re\sp d$ is an optimal path for the
optimal control problem with initial position $(y,s)$, then there
exists a \emph{dual} function $\lambda:[s,T]\rightarrow \Re\sp d$,
such that $x$ and $\lambda$ solve
\begin{equation}\label{eq:HamSyst}
\begin{split}
x'(t)&=H_\lambda\big(x(t),\lambda(t)\big), \quad\text{for } s < t <
T,\\
x(s)&=y,\\
-\lambda'(t)&=H_x(x(t),\lambda(t)\big), \quad\text{for } s < t <
T,\\
\lambda(T)&=g'\big(x(T)\big).
\end{split}
\end{equation}
This is the \emph{Pontryagin principle}, a necessary condition for
optimality, in the case where the Hamiltonian function, $H$, is
differentiable. The Hamiltonian is, however, often nondifferentiable,
even in situations where the control and cost functions in the control
problem are differentable, see \cite{Sandberg-Szepessy}. 
Even though the Pontryagin principle may be formulated for 
other control problems than those having differentiable Hamiltonians, 
the version in \eqref{eq:HamSyst} is appealing as a starting point for
numerical methods. 

In the Symplectic Pontryagin method, the Hamiltonian system
\eqref{eq:HamSyst} is used with a regularized Hamiltonian,
$H\sp\delta$,
satisfying $|H\sp\delta(x,\lambda)-H(x,\lambda)|\leq \delta$, for all
$(x,\lambda)\in \Re\sp {2d}$. The Hamiltonian system
\eqref{eq:HamSyst} makes sense when $H$ has been replaced with the
differentiable $H\sp\delta$. The comparison principle for
Hamilton-Jacobi equations then grants that the maximum difference
between the solution to the original Hamilton-Jacobi equation
\eqref{eq:HJ} and the one where  $H\sp\delta$ has taken the place of
$H$ is of the order $\delta$.

The second ingredient in the Symplectic Pontryagin method is the
application of the \emph{Symplectic Euler} numerical scheme for the
Hamiltonian system \eqref{eq:HamSyst} with the regularized Hamiltonian
$H\sp\delta$. We shall for simplicity assume that an approximation of
$u(x_s,0)$ is to be computed. The method to approximate $u$ for a
starting position whose time coordinate differs from zero, is similar.
The time interval $[0,T]$ is split into $N$ intervals of length 
$\Delta t=T/N$. We introduce the notation $t_n=n\Delta t$, for
$n=0,\ldots,N$. The Symplectic Pontryagin scheme is:
\begin{equation}\label{eq:SymplecticPontryagin}
\begin{split}
x_{n+1}&=x_n+\Delta t H\sp\delta_\lambda(x_n,\lambda_{n+1}),\quad
n=0,\ldots,N-1 \\
x_0&=x_s,\\
\lambda_n&=\lambda_{n+1}+\Delta t H\sp\delta(x_n,\lambda_{n+1}),\quad n=0,\ldots,N-1 \\
\lambda_N&=g'(x_N).
\end{split}
\end{equation} 

In \cite{Sandberg-Szepessy} this method is analyzed by extending the
solutions $\{x_n\}$ to piecewise linear functions. These functions are
used to define an approximate value function which is shown to solve
a Hamilton-Jacobi equation equal to the initial equation, but with an
additional error term. The above mentioned comparison principle gives
the difference between the approximate and the exact value functions.

In this paper the Symplectic Pontryagin method is analyzed in a
different way. The first step is here to consider the Hamilton-Jacobi
equations with the original Hamiltonian $H$, and the regularized
variant $H\sp\delta$. As mentioned above, this gives a difference of
the order $\delta$ between the corresponding solutions. Next, a
representation formula of the solutions to the Hamilton-Jacobi
equations as minima of a variation problem is used.  It is shown that
there exists one 
solution to the Symplectic Pontryagin method which is also a minimizer to
an Euler  discretized version of the variation problem. The minimizer
to the dicretized variation problem is then shown to give a value
which is close to the value of the original variation problem.

This work extends the result in \cite{Sandberg-Szepessy} in the
following ways:
\begin{itemize}
\item A solution to the Symplectic Pontryagin method is shown to exist
   for the considered class of Hamiltonians.
\item The result in \cite{Sandberg-Szepessy} relies on the assumption
  that the variation $\partial\lambda_{n+1}/\partial x_n$ is bounded
  everywhere but on a codimension one hypersurface. This assumption is
  not needed with the present approach.
\item The error bound in the present paper is shown to be of the order
  $\delta + \Delta t$, compared with the previous $\delta+\Delta
  t+\Delta t\sp 2 /\delta$. The new result therefore ensures the
  possibility of decreasing $\delta$ independently of $\Delta t$,
  without deteriorating the error estimate.\end{itemize}

The paper is organized as follows. In section \ref{sec:Main} we present the main
results of this paper, existence  of solutions to, and convergence of,
the Symplecitic Pontryagin method. In sections \ref{sec:low} and
\ref{sec:low2} the convergence proof is given in the form of a series
of lemmas. 
\section{Main Results}\label{sec:Main}
Given the Hamiltonian function $H:\Re\sp d\times\Re\sp d\rightarrow\Re$ we define
the \emph{running cost} function
\begin{equation}\label{eq:LegendreL}
L(x,\alpha)=\sup_{\lambda\in\Re\sp d}
\big\{-\alpha\cdot\lambda+H(x,\lambda)\big\},
\end{equation}
for all $x$ and $\alpha$ in $\Re\sp d$.
This function is convex in its second argument, and extended valued,
i.e.\ its values belong to $\Re \cup \{+\infty\}$. If the Hamiltonian
is real-valued and concave in its second variable it is possible to
get it back when having possession of L:
\begin{equation}\label{eq:LegendreH}
H(x,\lambda)=\inf_{\alpha\in\Re\sp d} \big\{\lambda\cdot\alpha+L(x,\alpha)\big\}.
\end{equation}
This is a consequence of the bijectivity of the Legendre-Fenchel
transform, see \cite{Rockafellar}.

For Hamilton-Jacobi equations with initial data given (Cauchy
problems, i.e.\ not as here with data given at time $T$) and convex
Hamiltonians, the Hamiltonian and the running cost are connected via
the usual Legendre-Fenchel transform. The results presented here could
have been presented for  Hamilton-Jacobi equations in this form,
but since it is more common for  optimal control problems  to include
cost functions of the terminal positions than the initial positions,
the current setting is chosen.

The cornerstone in the convergence analysis for the Symplectic
Pontryagin method will be the following representation theorem, taken
from \cite{Galbraith}, see also 
\cite{Rockafellar,Cannarsa-Sinestrari,Ishii88,Stromberg07}. The result is given in greater generality in
that paper, but we present a form suited for our purposes. 
We will assume that the Hamiltonian satisfies the following estimates:
\begin{equation}\label{eq:Hamiltonianestimates}
\begin{split}
|H(x,\lambda_1)-H(x,\lambda_2)|&\leq C_1|\lambda_1-\lambda_2|,\quad
 \text{for all }x,\lambda_1,\lambda_2 \in \Re\sp d, \\
|H(x_1,\lambda)-H(x_2,\lambda)|&\leq C_2|x_1-x_2|(1+|\lambda|),\quad
 \text{for all }x_1,x_2, \lambda \in\Re\sp d.
\end{split}
\end{equation}
Then the following representation result holds:
\begin{thm}\label{thm:representation}
Let the Hamiltonian $H:\Re\sp d\times\Re\sp d\rightarrow \Re$ be
concave in its second argument, and satisfy the bounds in 
\eqref{eq:Hamiltonianestimates}. 
Let the running cost $L$ be defined by \eqref{eq:LegendreL}, and 
let $g:\Re\sp d\rightarrow \Re$ be a
continuous function such that $g(x) \geq -k(1+|x|)$ for all
$x\in\Re\sp d$ for some $k > 0$. 
Then the value function
\begin{multline}\label{eq:V}
V(y,s)=\inf\Big(\int_s\sp T L\big(x(t),x'(t)\big)dt + g\big(x(T)\big)\ 
\big|\\ 
x:[s,T]\rightarrow\Re\sp d\text{ absolutely continuous,
}x(s)=y\Big)
\end{multline}
is a continuous viscosity solution to \eqref{eq:HJ} which satisfies
\begin{equation}\label{eq:existbound}
u(x,t) \geq -k(1+|x|)\quad \text{for all } (x,t)\in\Re\sp d\times[0,T].
\end{equation}

Furthermore, if a function is a continuous viscosity solution to
\eqref{eq:HJ} which satisfies \eqref{eq:existbound}, then this
function must be the value function $V$.
\end{thm}

We also have the following result, stating that optimal solutions to
the value function $V$ solve a Hamiltonian system. For a proof, see
\cite{Clarke}.
\begin{thm}\label{thm:optsolsolvehamsyst}
Let the conditions in Theorem \ref{thm:representation} be
satisfied, and let $(y,s)$ be any point in $\Re\sp d \times[0,T]$. 
Then a minimizer $x:[s,T]\rightarrow\Re\sp d$ for $V(y,s)$ in
\eqref{eq:V} exists. If furthermore the Hamiltonian $H$ and the
terminal cost $g$ are
continuously differentiable, there exists a dual function
$\lambda:[s,T]\rightarrow\Re\sp d$, such that $(x,\lambda)$ solve the
Hamiltonian system \eqref{eq:HamSyst}  
\end{thm}

In view of Theorem \ref{thm:representation} a possible approximate
value of $u(x_s,t_m)$ is the minimizer of
\begin{equation}\label{eq:J}
J_{(x_s,t_m)}(\alpha_m,\ldots,\alpha_{N-1}) =
\Delta t\sum_{n=m}\sp{N-1} L(x_n,\alpha_n) +g(x_N)
\end{equation}
where $\alpha_n\in \Re\sp d$ and 
\begin{equation}\label{eq:xalpharel}
\begin{split}
x_{n+1} &= x_n +\Delta t\alpha_n,\quad\text{for }m \leq n\leq N-1,\\
x_m&=x_s.
\end{split}
\end{equation}
We define a discrete value function as the optimal value of $J$:
\begin{equation*}
\bar u(x_s,t_m)=\inf\big\{J_{(x_s,t_m)}(\alpha_m,\ldots,\alpha_{N-1})\big\|\ \alpha_m,\ldots,\alpha_{N-1}\in\Re\sp d\big\}
\end{equation*}

For the proof of  Theorem \ref{thm:OptimEquivSP}, the existence
theorem, we need two definitions.
\begin{definition}\label{def:semiconcavity} 
A function $f:\Re\sp d\rightarrow\Re$ is \emph{locally semiconcave} if
for every compact convex set $V \subset \Re\sp d$ there exists a
constant $K>0$ such that $f(x)-K|x|\sp 2$ is concave on $V$.
\end{definition}
There exist alternative definitions of semiconcavity, see
\cite{Cannarsa-Sinestrari}, but this is the one used in this paper.
\begin{definition}\label{def:superdifferential}
An element $p\in\Re\sp d$ belongs to the \emph{superdifferential} of
the function $f:\Re\sp d\rightarrow\Re$ at $x$, denoted $D\sp+ f(x)$, if 
\begin{equation*}
\limsup_{y\rightarrow x} \frac{f(y)-f(x)-p\cdot(y-x)}{|y-x|} \leq 0.
\end{equation*}
\end{definition}

\begin{thm}\label{thm:OptimEquivSP}
Let $x_s$ be any element in $\Re\sp d$, and $g:\Re\sp d\rightarrow\Re$
a continuously differentiable function such that $|g'(x)|\leq L_g$,
for all $x\in\Re\sp d$, and some constant $L_g > 0$.
Let the Hamiltonian $H:\Re\sp d\times\Re\sp d\rightarrow \Re$ be two times
continuously differentiable and concave in its second argument, and satisfy the bounds in
\eqref{eq:Hamiltonianestimates}. Let the running cost $L$ be defined
by
\eqref{eq:LegendreL}. Then there exists a minimizer
$(\alpha_m,\ldots,\alpha_{N-1})$ of the function $J_{(x_s,t_m)}$ in \eqref{eq:J}. Let
$(x_m,\ldots,x_N)$ be a corresponding solution to \eqref{eq:xalpharel}.
Then there exists a discrete dual variable
$\lambda_n,\ n=m,\ldots,N$ satisfying 
\begin{equation}\label{eq:SymplPontr}
\begin{split}
x_{n+1}&=x_n+\Delta t H_\lambda(\lambda_{n+1},x_n), \quad\text{for all
} m \leq n\leq N-1,\\
x_m&=x_s\\
\lambda_n &= \lambda_{n+1}+\Delta t
H_x(\lambda_{n+1},x_n),\quad\text{for all }m \leq n\leq N-1,\\
\lambda_N&=g'(x_N).
\end{split}
\end{equation}
Hence $\alpha_n=H_\lambda(\lambda_{n+1},X_n)$ for all $m \leq n\leq N-1$.
\end{thm}
\begin{proof}

\emph{Step 1.} By the properties of the Hamiltonian it follows that
the running cost $L$ is lower semicontinuous, see
\cite{Rockafellar}. It is clear that the infimum of $J_{(x_s,t_m)}$ is
less than infinity. Let $\varepsilon > 0$ and
$\alpha\sp\varepsilon_n$, $n=m,\ldots,N-1$ be discrete controls such that
\begin{multline*}
J_{(x_s,t_m)}(\alpha\sp\varepsilon_m,\ldots,\alpha\sp\varepsilon_{N-1})\\
\leq \inf\big\{J_{(x_s,t_m)}(\alpha_m,\ldots,\alpha_{N-1})\ \big|\
\alpha_n\in\Re\sp d,\ n=m,\ldots,N-1\big\}+\varepsilon.
\end{multline*}
Since the Hamiltonian $H$ satisfies the bounds
\eqref{eq:Hamiltonianestimates}, it follows that all
$\alpha\sp\varepsilon_n$ are bounded  by $C_1$. Hence there exist
$\{\varepsilon_i\}$ and corresponding $\{\alpha_n\sp{\varepsilon_i}\}$ such that
$\varepsilon_i \rightarrow 0$, and $\alpha_n\sp{\varepsilon_i}\rightarrow\alpha_n$,
where $|\alpha_n|\leq C_1$, for all $n$. The lower semicontinuity of
$L$ implies that this $\{\alpha_n\}_{n=m}\sp{N-1}$ is a minimizer for $J_{(x_s,0)}$.

\emph{Step 2.} Assume that $\bar u(\cdot,t_{n+1})$ is locally semiconcave, and
that $\lambda_{n+1}\in D\sp +\bar u(x_{n+1},t_{n+1})$. We will
show that this implies
\begin{equation}\label{eq:HLconnection}
\lambda_{n+1}\cdot\alpha_n + L(x_n,\alpha_n)=H(\lambda_{n+1},x_n).
\end{equation}

By the semiconcavity we have that there exists a constant $A>0$, such that
\begin{equation}\label{eq:onestepvalue}
\bar u(x_n+\Delta t \alpha,t_{n+1}) \leq \bar u(x_{n+1},t_{n+1}) +
\Delta t \lambda_{n+1}\cdot(\alpha-\alpha_n) + A|\alpha-\alpha_n|\sp 2.
\end{equation}
for all $\alpha$ in a neighborhood around $\alpha_n$.
Since we know that the function 
\begin{equation*}
\alpha\mapsto \bar u(x_n+\Delta t \alpha,t_{n+1})+\Delta t L(x_n,\alpha)
\end{equation*}
is minimized for $\alpha=\alpha_n$, the semiconcavity of $\bar u$ in 
\eqref{eq:onestepvalue} implies that  the function 
\begin{equation}\label{eq:funcalmostvalue}
\alpha\mapsto \lambda_{n+1}\cdot\alpha + A|\alpha-\alpha_n|\sp 2 +L(x_n,\alpha)
\end{equation}
is also minimized for $\alpha=\alpha_n$. We will prove that 
the function 
\begin{equation}\label{eq:funcaffine}
\alpha\mapsto \lambda_{n+1}\cdot\alpha + L(x_n,\alpha)
\end{equation}
is minimized for $\alpha=\alpha_n$. Let us assume that this is false,
so that there exists an $\alpha\sp *\in\Re\sp d$, and an $\varepsilon > 0$ such that
\begin{equation}\label{eq:contradictionassumption}
\lambda_{n+1}\cdot\alpha_n+L(x_n,\alpha_n) -
\lambda_{n+1}\cdot\alpha\sp*-L(x_n,\alpha\sp *) \geq \varepsilon.
\end{equation}
Let $\xi\in[0,1]$, and $\hat\alpha=\xi\alpha\sp *
+(1-\xi)\alpha_n$. Insert $\hat \alpha$ into the function in
\eqref{eq:funcalmostvalue}:
\begin{multline*}
\lambda_{n+1}\cdot\hat\alpha+A|\hat\alpha-\alpha_n|\sp
2+L(x_n,\hat\alpha)\\
=
\xi\lambda_{n+1}\cdot\alpha\sp*+(1-\xi)\lambda_{n+1}\cdot\alpha_n+A\xi\sp
2|\alpha\sp*-\alpha_n|\sp 2+L(x_n,\xi\alpha\sp *+(1-\xi)\alpha_n)\\ 
\leq \xi\lambda_{n+1}\cdot\alpha\sp*+(1-\xi)\lambda_{n+1}\cdot\alpha_n+A\xi\sp
2|\alpha\sp*-\alpha_n|\sp 2 + \xi L(x_n,\alpha\sp *)
+(1-\xi)L(x_n,\alpha_n)\\
\leq \lambda_{n+1}\cdot\alpha_n + L(x_n,\alpha_n) + A\xi\sp
2|\alpha\sp*-\alpha_n|\sp 2 -\xi\varepsilon \\
<\lambda_{n+1}\cdot\alpha_n + L(x_n,\alpha_n),
\end{multline*}
for some small positive  number $\xi$. This contradicts the fact that
$\alpha_n$ is a minimizer to the function in
\eqref{eq:funcalmostvalue}. Hence we have shown that the function in
\eqref{eq:funcaffine} is minimized at $\alpha_n$. By the relation
between $L$ and $H$ in \eqref{eq:LegendreH} our claim
\eqref{eq:HLconnection} follows.

\emph{Step 3.} From the result in Step 2, equation
\eqref{eq:HLconnection}, and the definition of the running cost $L$ in
\eqref{eq:LegendreL} it follows that
$\alpha_n=H_\lambda(x_n,\lambda_{n+1})$, for if this equation did not
hold, then $\lambda_{n+1}$ could not be the maximizer of 
$-\alpha_n\cdot\lambda+H(x_n,\lambda)$.

\emph{Step 4.} In this step, we  show that if $\bar
u(\cdot,t_{n+1})$ is locally semiconcave, and $\lambda_{n+1}\in D\sp + \bar
u(x_{n+1},t_{n+1})$, then $\bar u(\cdot,t_n)$ is locally semiconcave, and
$\lambda_n\in D\sp + \bar u(x_n,t_n)$. 

By the assumed local semiconcavity of $\bar u(\cdot,t_{n+1})$ we have
that there exists a $K>0$ such that 
\begin{equation*}
\bar u (x,t_{n+1}) \leq \lambda_{n+1}\cdot(x-x_{n+1})+K|x-x_{n+1}|\sp 2,
\end{equation*}
for all $x$ in some neighborhood around $x_{n+1}$.
Let us now consider the control
$H_\lambda(x,\lambda_{n+1})$ at the point $(x,t_n)$. Since this
control is not necessarily optimal except at $(x_n,t_n)$, we have
\begin{multline}\label{eq:semicon1}
\bar u (x,t_n) \leq \bar u\big(x+\Delta t
H_\lambda(x,\lambda_{n+1}),t_{n+1}\big) +\Delta t
L\big(x,H_\lambda(x,\lambda_{n+1})\big) \\
\leq \bar u(x_{n+1},t_{n+1})+\lambda_{n+1}\cdot\big(x+\Delta t
H_\lambda(x,\lambda_{n+1})-x_{n+1}\big)\\
+ K|x+\Delta t
H_\lambda(x,\lambda_{n+1})-x_{n+1}|\sp 2 + \Delta t
L\big(x,H_\lambda(x,\lambda_{n+1})\big), 
\end{multline}
for all $x$ in some neighborhood around $x_n$.
By the definition of $L$ in \eqref{eq:LegendreL} it follows that 
\begin{equation*}
L\big(x,H_\lambda(x,\lambda_{n+1})\big)=-H_\lambda(x,\lambda_{n+1})\cdot\lambda_{n+1}+H(x,\lambda_{n+1}).
\end{equation*} 
With this fact in \eqref{eq:semicon1} we have
\begin{multline}\label{eq:semicon2}
\bar u (x,t_n)
\leq \bar u(x_{n+1},t_{n+1})+\lambda_{n+1}\cdot(x-x_{n+1})\\
+ K|x+\Delta t
H_\lambda(x,\lambda_{n+1})-x_{n+1}|\sp 2 + \Delta t H(x,\lambda_{n+1}). 
\end{multline}
By the result in step 3 we have that
\begin{equation*}
x_{n+1}=x_n+\Delta t H_\lambda(x_n,\lambda_{n+1}),
\end{equation*}
so that, by the fact that $H$ is twice continuously differentiable
\begin{multline}\label{eq:semicon3}
|x+\Delta t H_\lambda(x,\lambda_{n+1})-x_{n+1}|\\ 
= |x-x_n + \Delta
 t\big(H_\lambda(x,\lambda_{n+1})-H_\lambda(x_n,\lambda_{n+1}) \big)| 
\leq K|x-x_n|,
\end{multline}
for some (new) $K$ and $x$ in a neighborhood of $x_n$. We
also need the facts that
\begin{equation}\label{eq:semicon4}
\bar u(x_n,t_n)=\bar u(x_{n+1},t_{n+1})+\Delta t L(x_n,\alpha_n)
\end{equation}
and
\begin{equation}\label{eq:semicon5}
L(x_n,\alpha_n)=-\lambda_{n+1}\cdot\alpha_n + H(x_n,\lambda_{n+1}) = 
-\lambda_{n+1}\cdot\frac{x_{n+1}-x_n}{\Delta t}+ H(x_n,\lambda_{n+1}).
\end{equation}
We insert the results \eqref{eq:semicon3}, \eqref{eq:semicon4}, and \eqref{eq:semicon5} into
\eqref{eq:semicon2}. This implies that there are constants $K>$ and
$K'>0$, such that 
\begin{multline*}
\bar u(x,t_n)\\ 
\leq \bar u(x_n,t_n)+\lambda_{n+1}\cdot(x-x_n)+\Delta t 
\big(H(x,\lambda_{n+1})-H(x_n,\lambda_{n+1})\big) +K|x-x_n|\sp 2 \\
\leq  \bar u(x_n,t_n)+\big(\lambda_{n+1}+\Delta t
H_x(x_n,\lambda_{n+1})\big)\cdot(x-x_n) +K|x-x_n|\sp 2\\
=\bar u(x_n,t_n)+\lambda_n\cdot(x-x_n)+K|x-x_n|\sp 2,
\end{multline*}
for all $x$ in a neighborhood around $x_n$.
We here used the differentiability of $H$ in the last inequality, and
the $\lambda_n$ evolution in \eqref{eq:SymplPontr} in the last
equality. 
This shows that $u(\cdot,t_n)$ is locally semiconcave, and that
$\lambda_n\in D\sp+ \bar u(x_n,t_n)$.

\emph{Step 5.} 
Since $\bar u(x,T)=g(x)$ and $g$ is differentiable, it follows that 
$\bar u(\cdot,T)$ is  semiconcave. 
Induction backwards in $n$ shows that $\bar u(\cdot,t_n)$ is locally
semiconcave for all $n$. Thereby the conclusions in step 2 and 3 hold
for all $n$.

\end{proof}

%

By the results from sections \ref{sec:low} and \ref{sec:low2}, Lemmas
\ref{lem:lowerbound} and \ref{thm:upperbound}  we have the following
convergence result.
\begin{thm}\label{thm:upplowconv}
Let the Hamiltonian $H$ be concave in its second argument, be
continuously differentiable and satisfy
\eqref{eq:Hamiltonianestimates}. Let $g:\Re\sp d\rightarrow\Re$ be
differentiable and satisfy $|g'(x)|\leq C_3$, for all $x\in\Re\sp
d$. Let $x_s$ be any element in $\Re\sp d$. Then the viscosity
solution $u$ satisfying the conditions in Theorem
\ref{thm:representation}, and the approximate value function $\bar u$
satisfy
\begin{multline*}
  -\frac{C_1C_2T}{2}\big((e\sp{C_2T}-1)\Delta t+\Delta t\sp
2\big)-\frac{C_1C_3}{2}(e\sp{C_2T}-1)\Delta t\\ 
\leq u(x_s,0)-\bar
u(x_s,0)\\
\leq \half C_1C_2(C_3+1)e\sp{C_2 T}T\Delta t.
\end{multline*}
\end{thm}

If the Hamiltonian is not differentiable, the following result may be
used.
\begin{thm}
Let $H$ be a Hamiltonian which is concave in its second argument, and
satisfies \eqref{eq:Hamiltonianestimates}. Let $H\sp\delta$ be a
Hamiltonian that satisfies the same conditions, and which is
continuously differentiable, and satisfies 
\begin{equation*}
\|H-H\sp\delta\|_{L\sp\infty(\Re\sp d\times \Re\sp d)}\leq \delta.
\end{equation*}
Then 
\begin{multline*}
  -\frac{C_1C_2T}{2}\big((e\sp{C_2T}-1)\Delta t+\Delta t\sp
2\big)-\frac{C_1C_3}{2}(e\sp{C_2T}-1)\Delta t-T \delta \\ 
\leq u(x_s,0)-\bar
u(x_s,0)\\
\leq \half C_1C_2(C_3+1)e\sp{C_2 T}T\Delta t+T\delta,
\end{multline*}
where $\bar u$ is computed using $H\sp\delta$.
\end{thm}
\begin{proof}
The comparison principle for solutions to Hamilton-Jacobi equations
gives that the two solutions to \eqref{eq:HJ} with $H$ and
$H\sp\delta$ differ by $T\delta$, see \cite{Sandberg-Szepessy}. The
result follows from Theorem \ref{thm:upplowconv}. 
\end{proof}

\section{Lower bound of $u(x_s,0)-\bar u(x_s,0)$}\label{sec:low}


In order to be able to prove the inequality, we need the result in
Theorem \ref{thm:Rockafellar} for convex functions. 
The theorem uses the
following definition.
\begin{definition}\label{def:dirrec}
Let $f$ be a convex real-valued function defined on $\Re\sp d$. 
A vector $y\neq 0$ is called a \emph{direction of recession} of $f$ if
$f(x+ty)$ is a nondecreasing function of $t$, for every $x\in\Re\sp d$.
\end{definition}
It is shown in \cite{Rockafellar} that if $f(x+ty)$ is a nondecreasing
function of $t$ for one $x\in\Re\sp d$, then this property holds for
all $x\in\Re\sp d$.
The following thorem is taken from \cite{Rockafellar}, but in a
slightly simplified form.
\begin{thm}\label{thm:Rockafellar}
Let $\big\{f_i\ |\ i \in I\big\}$, where $I$ is an arbitrary index
set, be a collection of real-valued
convex functions on $\Re\sp d$ 
which have no common direction of recession.
Then one and only one of the following alternatives holds:
\begin{enumerate}[a)]
\item There exists a vector $x\in\Re\sp d$ such that 
\begin{equation*}
f_i(x) \leq 0,\quad \text{for all } i\in I.
\end{equation*}
\item There exist non-negative real numbers $\l_i$, only finitely
  many non-zero, such that, for some $\varepsilon > 0$, one has
\begin{equation}\label{eq:lambdaf}
\sum_{i\in I} \l_i f_i(x) \geq \varepsilon,\quad \text{for all }
x\in \Re\sp d.
\end{equation}
\end{enumerate}
\end{thm}

\begin{lemma}\label{lem:alphaexistence}
Assume that the Hamiltonian $H$ satisfies equation
\eqref{eq:Hamiltonianestimates} and is concave in its second argument. Let $x_1,\ x_2,$ and $\alpha_1$ be any
elements in $\Re\sp d$, and $\beta$ any positive number. Then there exists an element $\alpha_2$ in
$\Re\sp d$ such that 
\begin{equation}\label{eq:existalphabound}
|\alpha_2-\alpha_1|\leq C_2|x_2-x_1|+\beta,
\end{equation}
and
\begin{equation}\label{eq:exitssupphyperplane}
L(x_2,\alpha_2) \leq L(x_1,\alpha_1)+C_2|x_2-x_1|.
\end{equation}
\end{lemma}
\begin{proof}
If $L(x_1,\alpha_1)=\infty$, we can let $\alpha_2=\alpha_1$. We
henceforth assume that $L(x_1,\alpha_1)<\infty$.
From the definition of $L$ in \eqref{eq:LegendreL} it follows that
$H(x_1,0)\leq L(x_1,\alpha_1)$. By \eqref{eq:Hamiltonianestimates} we
have 
\begin{equation}\label{eq:HLdiff}
H(x_2,0) \leq L(x_1,\alpha_1)+C_2|x_2-x_1|.
\end{equation}
We start by assuming that a strict inequality holds in
\eqref{eq:HLdiff}, and then use the result for this case to prove the
result in the general situation where we also allow equality in the
equation.

When we assume that $H(x_2,0) > L(x_1,\alpha_1)+C_2|x_2-x_1|$, it is
possible to use Theorem \ref{thm:Rockafellar}  with the functions
\begin{equation}\label{eq:collconvfunc}
\lambda \mapsto -H(x_2,\lambda)+\alpha\cdot\lambda+L(x_1,\alpha_1)+C_2|x_2-x_1|.
\end{equation}
We let $\alpha$ be any element in the ball of radius  
$C_2|x_2-x_1|+\beta$ centered at $\alpha_1$, the set we will use as the index
set $I$ in Theorem \ref{thm:Rockafellar}.

We now show that the functions in \eqref{eq:collconvfunc} have
no  common direction of recession. Let $\lambda=a v$, where $v$
is a unit 
vector in $\Re\sp d$, and let
$\alpha=\alpha_1+\big(C_2|x_2-x_1|+\beta\big)v$. The corresponding
function from \eqref{eq:collconvfunc} then satisfies
\begin{equation}\label{eq:infconv}
\begin{split}
&-H(x_2,a v)+a\alpha_1\cdot v+
\big(C_2|x_2-x_1|+\beta\big)a+L(x_1,\alpha_1)+C_2|x_2-x_1|\\ 
&=
-H(x_1,a v)+a\alpha_1\cdot v +\big(C_2|x_2-x_1|+\beta\big)a +H(x_1, a v)-H(x_2,a
v) \\
&\quad +L(x_1,\alpha_1) +C_2|x_2-x_1|\geq \beta a  \rightarrow \infty\quad\text{when
}a\rightarrow \infty,
\end{split}
\end{equation}
since $L(x_1,\alpha_1) \geq -a\alpha_1\cdot v +H(x_1,av)$ by the
definition of $L$ in \eqref{eq:LegendreL}, and 
$H(x_1,a\cdot v)-H(x_2,a\cdot v) \geq -C_2|x_2-x_1|(1+a)$ by equation
\eqref{eq:Hamiltonianestimates}.

We can therefore apply Theorem \ref{thm:Rockafellar} on the set of
functions in \eqref{eq:collconvfunc} with the $\alpha$:s taken from
the ball of radius $C_2|x_2-x_1|+\beta$, centered at $\alpha_1$. 
Since all functions in \eqref{eq:collconvfunc} are positive at $\lambda=0$,
and at least one function is positive at every other point, by
\eqref{eq:infconv}, alternative a) in Theorem \ref{thm:Rockafellar}
can not hold. 
Therefore the set of functions satisfy alternative b). 
We may assume that the numbers $\l_i$ in \eqref{eq:lambdaf}
satisfy 
\[
\sum_{i\in I} l_i =1.
\]
This corresponds to a multiplication of the inequality
\eqref{eq:lambdaf}  by a positive number. This changes the number
$\varepsilon$, but that is not important here. Hence we have that there
exist a finite number of vectors $\alpha\sp i$, with 
\begin{equation}\label{eq:alphaball}
|\alpha\sp
i-\alpha_1| \leq C_2|x_2-x_1|+\beta, 
\end{equation}
such that
\begin{multline*}
\sum_i l_i\big(-H(x_2,\lambda)+\alpha\sp i \cdot
\lambda+L(x_1,\alpha_1)+C_2|x_2-x_1|\big)\\ 
=
-H(x_2,\lambda)+\big(\sum_i l_i\alpha\sp i\big)\cdot\lambda
+L(x_1,\alpha_1)+C_2|x_2-x_1| \geq \varepsilon.
\end{multline*}
Since every vector $\alpha\sp i$ satisfy \eqref{eq:alphaball}, so does
the convex combination
\begin{equation*}
\sum_i l_i\alpha\sp i.
\end{equation*}
This convex combination can be taken as the $\alpha_2$ in the lemma.

Let us now consider the remaining case where
$H(x_2,0)=L(x_1,\alpha_1)+C_2|x_2-x_1|$. We now modify the functions in
\eqref{eq:collconvfunc} to be
\begin{equation}\label{eq:collconvfunc2}
\lambda \mapsto
-H(x_2,\lambda)+\alpha\cdot\lambda+L(x_1,\alpha_1)+C_2|x_2-x_1|
+ \gamma,
\end{equation}
where $\gamma$ is a positive number. The same analysis as for
\eqref{eq:collconvfunc} shows that there is a vector $\alpha\sp\gamma$
whose distance from $\alpha_1$ is less than or equal to
$C|x_2-x_1|+\beta$, such that 
\begin{equation*}
-H(x_2,\lambda)+\alpha\sp\gamma\cdot\lambda
+L(x_1,\alpha_1)+C_2|x_2-x_1|+\gamma \geq 0.
\end{equation*}
Since all $\alpha\sp\gamma$ are contained in a compact set it is
possible to find a sequence $\{\gamma_n\}_1\sp\infty$ converging to
zero such that $\alpha\sp{\gamma_n}$ converges to an element we may
call $\alpha_2$. It is straightforward to check that this $\alpha_2$
satisfies the conditions in the lemma.
\end{proof}
The idea is now to use Lemma \ref{lem:alphaexistence} in order to show
that there on each interval $(t_n,t_{n+1})$ exists a function
$\tilde\alpha(t)$, which satisfies
\begin{equation*}
\begin{split}
|\tilde\alpha(t)-\alpha(t)| &\leq C_2|x_n-x(t)| + \beta,\\
L\big(x_n,\tilde\alpha(t)\big)& \leq L\big(x(t),\alpha(t)\big) + C_2|x_n-x(t)|,
\end{split}
\end{equation*}
where the function $x(t)$ is a minimizer to the value function V in
\eqref{eq:V}. While Lemma \ref{lem:alphaexistence} provides the
existence of such an $\tilde\alpha(t)$ for each particular time $t$,
it does not provide us with a measurable function $\tilde\alpha$. In
order to show that such a measurable function exists, the result in
Theorem \ref{thm:Michael}, Michael's Selection Theorem, is needed. It can be found in a
more general form in e.g.\ \cite{Aubin-Cellina}. Before stating it we
need the following definition, which can also  be found in a more
general form in \cite{Aubin-Cellina}.
\begin{definition}\label{def:Lowersemicontinuity}
A function $F$ from the interval $(t_1,t_2)$ into the nonempty subsets
of $\Re\sp d$ is \emph{lower semicontinuous} at $t\sp *\in(t_1,t_2)$
if for any $z\sp * \in F(t\sp *)$ and any neighborhood $N(z\sp *)$ of
$z\sp *$, there exists a neighborhood $N(t\sp *)$ of $t\sp *$ such
that for all $t\in N(t\sp *)$,
\begin{equation*}
F(t)\cap N(z\sp *) \neq \emptyset.
\end{equation*}
We say that $F$ is \emph{lower semicontinuous} if it is lower
semicontinuous at every $t\in(t_1,t_2)$.
\end{definition}
\begin{thm}\label{thm:Michael}
Let the set-valued function $F$ from the interval $(t_1,t_2)$ into the
closed convex subsets of $\Re\sp d$ be lower semicontinuous. Then
there exists $f:(t_1,t_2)\rightarrow\Re\sp d$ which is a continuous
selection of $F$, i.e.\ which satisfies $f(t)\in F(t)$ for all $t_1<t<t_2$.
\end{thm}
 
\begin{lemma}\label{lem:lowersemicont}
Let $\beta$ be any positive number, $x_n$ any element in $\Re\sp d$,  the terminal cost $g$ and the Hamiltonian $H$ be  continuously
differentiable, and $H$ be concave in its second argument and satisfy \eqref{eq:Hamiltonianestimates}. 
Let $x:[0,T]\rightarrow\Re\sp d$ be a minimizer for the value
$V(x_s,0)$, defined in \eqref{eq:V}.
Let $A(t)$ for $t_n < t < t_{n+1}$ be the set of  all $\alpha\in
\Re\sp d$, such that the following conditions are satisfied:
\begin{equation}\label{eq:alphabound}
|\alpha-x'(t)|\leq C_2|x_n-x(t)|+\beta,
\end{equation}
and
\begin{equation}\label{eq:lambdabound}
H(x_n,\lambda)\leq
\alpha\cdot\lambda+L\big(x(t),x'(t)\big)+C_2|x_n-x(t)|,\quad  \text{for all } \lambda\in\Re\sp d.
\end{equation}
The set-valued function $A(t)$ is lower semicontinuous.
\end{lemma}
\begin{proof}
We use the notation $\alpha(t)=x'(t)$. By Theorem
\ref{thm:optsolsolvehamsyst} we know that $x$ solves the Hamiltonian
system \eqref{eq:HamSyst}. Hence $\alpha$ is continuous.
Let $t_n < t\sp * <t_{n+1}$ and $\alpha\sp *$ be an element in $A(t\sp
*)$. 
Similarly as in the proof of Lemma \ref{lem:alphaexistence}, we have
that
\begin{equation}\label{eq:ineqtwocases}
H(x_n,0) \leq L\big(x(t\sp *),\alpha(t\sp *)\big) + C_2|x_n-x(t\sp *)|.
\end{equation} 
We first consider the case where there is equality in \eqref{eq:ineqtwocases}, and
later the case of strict inequality.

\emph{Case 1.} We assume here that 
\begin{equation}\label{eq:equalityassumption}
H(x_n,0)= L\big(x(t\sp *),\alpha(t\sp *)\big)+C_2|x_n-x(t\sp *)|,
\end{equation} 
and let $t_n < t<t_{n+1}$.
We know by Lemma \ref{lem:alphaexistence} that $A(t)$ is nonempty. In order for equation
\eqref{eq:lambdabound} to be satisfied for an element in $A(t)$, we
must have
\begin{equation}\label{eq:Lminor}
L\big(x(t),\alpha(t)\big)+C_2|x_n-x(t)| \geq H(x_n,0)=L\big(x(t\sp
 *),\alpha(t\sp *)\big)+C_2|x_n-x(t\sp *)|.
\end{equation}

We now show that there exists an element in $A(t)$ 
which is close to $\alpha\sp *$ when $|t-t\sp *|$ is small.
First note that since $\alpha\sp *$ satisfies
\eqref{eq:lambdabound} at $t\sp *$, and we have assumption
\eqref{eq:equalityassumption}, it follows that $\alpha\sp
*=H_\lambda(x_n,0)$. By equation \eqref{eq:Lminor} it follows
that $\alpha\sp *$ satisfies
\eqref{eq:lambdabound}
also at $t$. The set of $\alpha$:s that solve \eqref{eq:lambdabound}
with $t$ fixed is convex. This convexity together with the fact that
$\alpha\sp *$ satisfies \eqref{eq:lambdabound}, and the existence of 
an element in $A(t)$, by Lemma \ref{lem:alphaexistence}, 
implies that there exists an element in $A(t)$ which is not farther
away from $\alpha\sp *$ than 
\begin{equation}\label{eq:maxdist}
\sqrt{|\alpha\sp *-\alpha(t)|\sp 2 -
  C_2\sp 2 (|x_n-x(t)|+\beta)\sp 2}, 
\end{equation}
see Figure \ref{fig:circle}.
\begin{figure}
\centering
\includegraphics[width=0.5\textwidth]{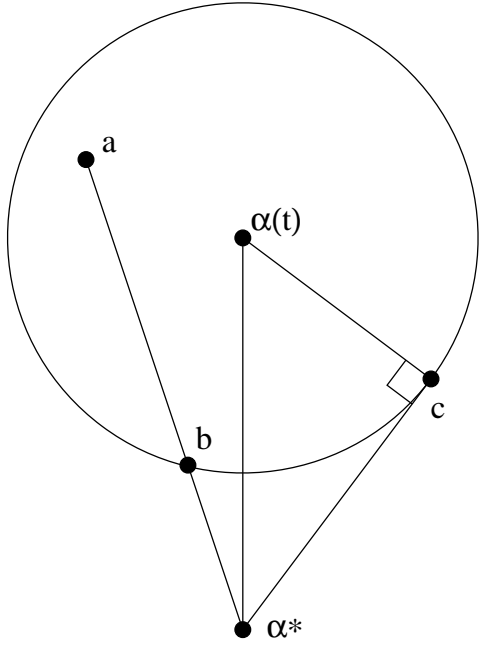}
\caption{The circle has radius $C_2|x_n-x(t)|+\beta$. When the dimension
  $d$ is two, the set $A(t)$ therefore
  consists of the points within the circle that satisfies
  \eqref{eq:lambdabound}. Since the set of points that satisfy this
  equation is convex, and $\alpha\sp *$ is one of them, it follows
  that if the point $a$ in the figure belongs to $A(t)$, then so does
  $b$. We therefore see that there exists a point in $A(t)$ that is no farther
  away from $\alpha\sp *$ than $c$, i.e.\ a distance $\sqrt{|\alpha\sp *-\alpha(t)|\sp 2 -
  C_2\sp 2 (|x_n-x(t)|+\beta)\sp 2}$.
  The situation in
  higher dimension is analogous.}
\label{fig:circle}
\end{figure}
We now use that 
\begin{equation*}
|\alpha(t)-\alpha\sp *| \leq |\alpha(t)-\alpha(t\sp *)| + |\alpha(t\sp
 *)-\alpha\sp *| \leq C_2|x_n-x(t\sp *)|+\beta + o(|t-t\sp *|)
\end{equation*}
and
\begin{equation*}
|x_n-x(t)| \geq |x_n-x(t\sp *)|-|x(t\sp *)-x(t)| \geq |x_n-x(t\sp *)|
- O(|t-t\sp *|)
\end{equation*}
in \eqref{eq:maxdist} to see that we can make the distance from
$\alpha\sp *$ to an element in $A(t)$ arbitrarily small if $|t-t\sp
*|$ is small. The case where the dimension $d=1$ can not be treated as
in Figure \ref{fig:circle}, but then the same line of reasoning shows
that there is an element in $A(t)$ which is a distance 
\begin{equation*}
\max\{|\alpha(t)-\alpha\sp*|-C|x_n-x(t)|,0\}
\end{equation*} 
from $\alpha\sp *$. It
follws that $A$ is lower semicontinuous at $t\sp *$.

\emph{Case 2.} We assume  here that $H(x_n,0) < L\big(x(t\sp *),\alpha(t\sp
*)\big) + C_2|x_n-x(t\sp *)|$.  As in the proof of Lemma
\ref{lem:alphaexistence} there exists an $\alpha\sp \#$ which satisfies 
$|\alpha\sp\# -\alpha(t\sp *)| \leq C_2|x_n-x(t\sp*)|+\beta$, and an
$\varepsilon > 0$ such that 
\begin{equation*}
\alpha\sp\# \cdot
\lambda+L\big(x(t\sp*),\alpha(t\sp*)\big)+C_2|x_n-x(t\sp*)| \leq
H(x_n,\lambda) + \varepsilon,\quad \text{for all } \lambda\in\Re\sp d.
\end{equation*} 
The linear combination $\xi\alpha\sp\#+(1-\xi)\alpha\sp*$, where
$0\leq \xi\leq 1$, satisfies
\begin{equation*}
\big(\xi\alpha\sp\#+(1-\xi)\alpha\sp*\big)\cdot\lambda +
L\big(x(t\sp*),\alpha(t\sp*)\big) + C_2|x_n-x(t\sp*)| \leq
H(x_n,\lambda)+\xi\varepsilon, 
\end{equation*}
for all $\lambda\in\Re\sp d.$
Along the optimal path  the running cost satisfies 
\[
L\big(x(t),\alpha(t)\big) =
H\big(x(t),\lambda(t)\big)-\lambda(t)\cdot H_\lambda\big(x(t),\lambda(t)\big).
\]
This follows by the fact that
$\alpha(t)=H_\lambda\big(x(t),\lambda(t)\big)$.
Since the Hamiltonian is assumed to be continuously differentiable, we
therefore have
\begin{equation*}
\Big|L\big(x(t),\alpha(t)\big)+C|x_n-x(t)|-\Big(L\big(x(t\sp
*),\alpha(t\sp *)\big)+C|x_n-x(t\sp *)|\Big)\Big| \leq o(|t-t\sp *|).
\end{equation*}
If $\xi$ is small enough, $\hat\alpha=\xi\alpha\sp\#+(1-\xi)\alpha\sp
*$ therefore satisfies
\begin{equation*}
\hat\alpha\cdot\lambda +L\big(x(t),\alpha(t)\big)+C_2|x_n-x(t)| \leq
H(x_n,\lambda), \quad \text{for all }\lambda\in\Re\sp d
\end{equation*}
Since $\hat\alpha$ satisfies equation \eqref{eq:lambdabound}, similar
reasoning  as in case 1 shows that the distance from $\alpha\sp *$ to
an element in $A(t)$ can be made arbitrarily small when $|t-t\sp *|$
is made small.
\end{proof}
We are now ready to state the error estimate.
\begin{lemma}\label{lem:lowerbound}
Let the Hamiltonian $H$ be continuously differentiable, concave in its
second argument, and satisfy
\eqref{eq:Hamiltonianestimates}, and the terminal cost $g$ be
continuously differentiable with $|g'(x)|\leq C_3$, for all
$x\in\Re\sp d$. Then
\begin{equation}\label{eq:thmlowerbound}
u(x_s,0)-\bar u(x_s,0) \geq -\frac{C_1C_2T}{2}\big((e\sp{C_2T}-1)\Delta t+\Delta t\sp
2\big)-\frac{C_1C_3}{2}(e\sp{C_2T}-1)\Delta t.
\end{equation}
\end{lemma}
\begin{proof}
As in Lemma \ref{lem:lowersemicont} we let $x:[0,T]\rightarrow\Re\sp
d$ be a minimizer for the value $V(x_s,0)$, defined in \eqref{eq:V},
and denote $x'(t)$ by $\alpha(t)$. It is easy to check that the
set-valued function $A(t)$ in Lemma \ref{lem:lowersemicont} has closed
and convex values. Hence Theorem \ref{thm:Michael} and Lemma
\ref{lem:lowersemicont} show that for every $\beta>0$ and any
$x_n\in\Re\sp d$, there exists a
measurable function   $\tilde\alpha(t)$, such that for $t_n < t
<t_{n+1}$
\begin{equation*}
\begin{split}
|\tilde\alpha(t)-\alpha(t)| &\leq C_2|x_n-x(t)|+\beta,\\
L\big(x_n,\tilde\alpha(t)\big)&\leq L\big(x(t),\alpha(t)\big)+C_2|x_n-x(t)|.
\end{split}
\end{equation*}
We now define the vectors $\{\tilde x_n\}_{n=0}\sp{N-1}$ and $\{\tilde\alpha_n\}_{n=0}\sp{N-1}$
as
follows. Let $\tilde x_0=x_s$. Then there exists a
$\tilde\alpha:(t_0,t_1)\rightarrow\Re\sp d$ as above. 
Let
\begin{equation*}
\tilde\alpha_0=\frac{1}{\Delta t}\int_{t_0}\sp{t_1} \tilde\alpha(t)dt,
\end{equation*}
and $\tilde x_1=\tilde x_0 + \Delta t\tilde\alpha_0$.
This $\tilde x_1$ now takes the role of $x_1$ in Lemma
\ref{lem:lowersemicont}. We can therefore find
$\tilde\alpha:(t_1,t_2)\rightarrow\Re\sp d$, and define
$\tilde\alpha_1$ and $\tilde x_2$ as before, and then iterate this
process further.

\emph{Step 1.} This step consists of a proof that 
\begin{equation}\label{eq:step1}
\max_{0\leq n\leq N}|x(t_n)-\tilde x_n|\leq(e\sp{CT}-1)\big(\frac{A}{2}\Delta t +\delta/C\big).
\end{equation} 
To
see this, we write 
\begin{equation}\label{eq:xdiffevol}
x(t_{n})-\tilde x_{n}=x(t_{n-1})-\tilde x_{n-1}+\int_{t_{n-1}}\sp{t_{n}}\big(\alpha(t)-\tilde\alpha(t)\big)dt.
\end{equation}
We now use that 
\begin{equation}\label{eq:ab}
|\alpha(t)-\tilde\alpha(t)| \leq C_2|\tilde x_{n-1}-x(t)|+\beta \leq C_2|x(t_{n-1})-\tilde x_{n-1}|
 + C_2|x(t)-x(t_{n-1})|+\beta,
\end{equation}
when $t_{n-1} \leq t <t_{n}$. We now use that $|x'(t)| \leq C_1$ by
\eqref{eq:Hamiltonianestimates}, and \eqref{eq:ab} in \eqref{eq:xdiffevol} to obtain
\begin{equation*}
\begin{split}
&|x(t_{n})-\tilde x_{n}| \leq (1+C_2\Delta t)|x(t_{n-1})-\tilde x_{n-1}|
 +\frac{C_1 C_2}{2}\Delta t\sp 2 +\beta\Delta t\\
 &\leq (1+C_2\Delta t)\sp 2 |x(t_{n-2})-\tilde x_{n-2}|+\big((1+C_2\Delta
 t)+1\big)\big(\frac{C_1 C_2}{2}\Delta t\sp 2+\beta\Delta t\big) \\
&\leq\ldots\leq (1+C_2\Delta t)\sp{n}|x(0)-\tilde x_0|\\ 
&+ 
\big((1+C_2\Delta t)\sp {n-1} + (1+C_2\Delta t)\sp{n-2}
 +\ldots+1\big)\big(\frac{C_1 C_2}{2}\Delta t\sp 2+\beta\Delta t\big)\\
&= \frac{(1+C_2\Delta t)\sp{n}-1}{C_2\Delta t}\big(\frac{C_1 C_2}{2}\Delta
 t\sp 2+\beta\Delta t\big) =\big((1+C_2\Delta
 t)\sp{n}-1\big)\big(\frac{C_1}{2}\Delta t+\beta/C_2\big)\\
&\leq e\sp{C_2 Tn/N}\big(\frac{C_1}{2}\Delta t+\beta/C_2\big) \leq 
(e\sp{C_2 T}-1)\big(\frac{C_1}{2}\Delta t +\beta/C_2 \big),
\end{split}
\end{equation*}
where the second last inequality follows from $\Delta
t=T/N$, and 
the second last equality follows from $x(0)=\tilde x_0$, and the
formula for a geometrical sum. 

\emph{Step 2.} We now provide a  lower bound for the term
$\int_0\sp T L\big(x(t),\alpha(t)\big)dt$, present in $u(x_s,0)$. For this
purpose we use that for $t_n\leq t <t_{n+1}$,
\begin{multline}\label{eq:Llowerbound}
L\big(\tilde x_n,\tilde\alpha(t)\big) \leq
L\big(x(t),\alpha(t)\big)+C_2|\tilde x_n-x(t)|\\
\leq L\big(x(t),\alpha(t)\big)+C_2|\tilde x_n-x(t_n)|+C_2|x(t)-x(t_n)|.
\end{multline}
From the boundedness of $|x'|\leq C_1$, we have 
\begin{equation}\label{eq:absxdiff}
\int_{t_n}\sp{t_{n+1}}|x(t)-x(t_n)|dt \leq \frac{C_1}{2}\Delta t\sp 2
\end{equation}
We now use the result from step 1, equation \eqref{eq:step1}, together
with \eqref{eq:Llowerbound} and \eqref{eq:absxdiff}:
\begin{multline}\label{eq:step2}
\int_0\sp T L\big(x(t),\alpha(t)\big)dt \geq
\sum_{n=0}\sp{N-1}\Big(\int_{t_n}\sp{t_{n+1}}L\big(\tilde x_n,\tilde\alpha(t)\big)dt\Big)\\
- C_2 T(e\sp{C_2 T}-1)\big(\frac{C_1}{2}\Delta t
+\beta/C_2\big)-\frac{C_1 C_2 T}{2}\Delta t\sp 2\\
\geq \Delta t\sum_{n=0}\sp{N-1}L(\tilde x_n,\tilde\alpha_n)-C_2
T(e\sp{C_2 T}-1)\big(\frac{C_1}{2}\Delta t
+\beta/C_2\big)-\frac{C_1 C_2 T}{2}\Delta t\sp 2
\end{multline}

\emph{Step 3.} Since 
\begin{equation*}
\bar u(x_s,0) \leq \Delta t\sum_{n=0}\sp{N-1} L(\tilde x_n,\tilde\alpha_n)+g(\tilde x_N)
\end{equation*} 
we can now achieve the desired lower bound, 
\begin{multline*}
u(x_s,0)-\bar u(x_s,0)\\ 
\geq \Delta t\sum_{n=0}\sp{N-1}L(\tilde x_n,\tilde\alpha_n)-C_2
T(e\sp{C_2 T}-1)\big(\frac{C_1}{2}\Delta t
+\beta/C_2\big)-\frac{C_1 C_2 T}{2}\Delta t\sp 2 +g\big(x(T)\big)\\ 
- \big(\Delta t\sum_{n=0}\sp{N-1} L(\tilde x_n,\tilde\alpha_n)+g(\tilde x_N)\big)\\
\geq -C_2 T(e\sp{C_2 T}-1)\big(\frac{C_1}{2}\Delta t
+\beta/C_2 \big)-\frac{C_1C_2T}{2}\Delta t\sp 2 - C_3 |x(T)-\tilde x_N|\\
\geq -C_2T(e\sp{C_2T}-1)\big(\frac{C_1}{2}\Delta t
+\beta/C_2\big)-\frac{C_1C_2T}{2}\Delta t\sp 2 - C_3(e\sp{C_2T}-1)\big(\frac{C_1}{2}\Delta t
+\beta/C_2\big).
\end{multline*}
Since the number $\beta$ can be made arbitrarily small, the lower
bound \eqref{eq:thmlowerbound} follows.
\end{proof}

\section{Lower bound of $\bar u(x_s,0)-u(x_s,0)$}\label{sec:low2}
In order to be able to prove this lower bound, we first need to
establish a bound on the discrete dual variable.

\begin{lemma}\label{lem:boundedlambda}
Suppose $g:\Re\sp d \rightarrow \Re$ is 
differentiable
and
satisfies $|g'(x)|\leq C_3$ for all $x\in\Re\sp d$. Suppose $H:\Re\sp d\times\Re\sp
d\rightarrow \Re$ is continuously differentiable and satisfies \eqref{eq:Hamiltonianestimates}.
Let $\{x_n,\lambda_n\}_{n=0}\sp{N}$ be an optimal discrete solution
and dual which satisfy \eqref{eq:SymplPontr}. 
Then 
\begin{equation}\label{eq:boundedlambda}
|\lambda_n|\leq (C_3+1) e\sp{C_2T}-1.
\end{equation}
\end{lemma}
\begin{proof}

\begin{multline*}
|\lambda_n|=|\lambda_{n+1}-\Delta t H_x (x_n,\lambda_{n+1})| \leq
 (1+C_2\Delta t)|\lambda_{n+1}|+C_2\Delta t\\
\leq\ldots\leq(1+C_2\Delta t)\sp{N-n}|\lambda_N|+\big((1+C_2\Delta
 t)\sp{N-n-1}+\ldots +1\big)C_2\Delta t\\
\leq C_3(1+C_2\Delta t)\sp{N}+\frac{(1+C_2\Delta t)\sp N-1}{C_2\Delta
 t}C_2\Delta t \\
\leq C_3 e\sp{C_2T}+e\sp{C_2 T}-1.
\end{multline*}
\end{proof}

\begin{lemma}\label{thm:upperbound}
Let the conditions in Theorem \ref{thm:representation} be
satisfied. Let $g:\Re\sp d\rightarrow\Re$ be differentiable and
satisfy $|g'(x)|\leq C_3$, for every $x\in\Re\sp d$. Let $H$ be
continuously differentiable. Then, for any $x_s\in\Re\sp d$, 
\begin{equation*}
\bar u(x_s,0)- u(x_s,0) \geq -\half C_1C_2(C_3+1)e\sp{C_2 T}T\Delta t.
\end{equation*}
\end{lemma}
\begin{proof}
The approximate value function is given by 
\begin{equation*}
\bar u(x_s,0)=\Delta t \sum_{n=0}\sp{N-1} L(x_n,\alpha_n)+g(x_N),
\end{equation*}
for  some $\{x_n,\alpha_n\}$.
The solution $\{x_n\}_{n=0}\sp{N}$ is extended to a piecewise linear
function $\bar x:[0,T]\rightarrow\Re\sp d$, defined by
\begin{equation*}
\bar x(t) = \frac{t-t_n}{\Delta t}x_{n+1}+\frac{t_{n+1}-t}{\Delta
  t}x_n, \quad\text{when }t_n\leq t \leq t_{n+1}.
\end{equation*}
Using this extended approximate solution, and the fact that  
\[
g(x_N)=u(x_N,T),
\] 
we can  represent the
difference between the original and the approximate value functions as
\begin{multline}\label{eq:upperstart}
\bar u(x_0,0)-u(x_0,0)=\Delta t\sum_{n=0}\sp{N-1}L(x_n,\alpha_n) +
u\big(\bar x(T),T\big)-u\big(\bar x(0),0\big) \\
=\Delta t\sum_{n=0}\sp{N-1}L(x_n,\alpha_n) +\int_0\sp T \frac{d}{dt}
u\big(\bar x(t),t\big)dt.
\end{multline}
The last equality in the above formula follows by the fact that the
function $t\mapsto u\big(\bar x(t),t\big)$ is absolutely continuous
since $x(t)$ is Lipschitz continuous, and $u$ is locally semiconcave.
That $u$ is locally semiconcave is shown in \cite{Cannarsa-Sinestrari}
for a slightly less general case than is considered here, but the
proof generalizes easily to the present conditions.

We now introduce the function $\tilde H:\Re\sp d\times\Re\sp
d\rightarrow\Re$, defined by
\begin{equation*}
\tilde H(x,\lambda) = H(x,\lambda)-
K\Big(\max\big(|\lambda|-((C_3+1)e\sp{C_2 T}-1),0\big)\Big)\sp 2,
\end{equation*}
where $K$ is a positive constant.
Define the corresponding running cost
\begin{equation}\label{eq:Ltilde}
\tilde L(x,\alpha)=\sup_{\lambda\in\Re\sp d}\big(-\alpha\cdot\lambda
+\tilde H(x,\lambda)\big).
\end{equation}
Because $\tilde H(x,\lambda) \leq H(x,\lambda)$ for all $x$ and
$\lambda$, it follows that $\tilde L(x,\alpha)\leq L(x,\alpha)$ for
all $x$ and $\alpha$. But  using Lemma \ref{lem:boundedlambda} we also
establish the opposite  inequality for the optimal state and control variables:
\begin{equation*}
\tilde L(x_n,\alpha_n) \geq -\alpha_n\cdot\lambda_{n+1} +\tilde
H(x_n,\lambda_{n+1}) =
-\alpha_n\cdot\lambda_{n+1}+H(x_n,\lambda_{n+1}) =L(x_n,\alpha_n).
\end{equation*}
Hence we may exchange $L$ with $\tilde L$ in \eqref{eq:upperstart}.

In order to bound the integral in \eqref{eq:upperstart} from above, we
now use the local semiconcavity of the value function $u$. This property
implies that the superdifferential $D\sp + u\big(\bar x(t),t\big)$ is
nonempty for all $0< t < T$. Let $p=(p_x,p_t)$ be an element in 
$D\sp + u\big(\bar x(t),t\big)$. 
Let $t\in(0,T)$ satisfy $t \neq t_n$ for $n=0,\ldots,N$, so that $\bar x$ is
differentiable at $t$. We split the difference quotient approximating
the backward derivative at $t$:
\begin{multline*}
\frac{u\big(\bar x(t),t\big)-u\big(\bar x(t-h),t-h\big)}{h}\\
=-\big[u\big(\bar x(t-h),t-h\big)-u\big(\bar
  x(t),t\big)-p_t(-h)-p_x\cdot\big(\bar x(t-h)-\bar x(t)\big)\big]/h
\\
+p_t+p_x\cdot\frac{\bar x(t)-\bar x(t-h)}{h}
\end{multline*}
The semiconcavity of $u$ implies that there is a constant $K$, such
that the quotient involving the square bracket above  is greater than
or equal to
\begin{equation*}
-K\big(h\sp 2+\big(\bar x(t)-\bar x(t-h)\big)\sp 2\big)/h.
\end{equation*}
When $h$ is small enough, the difference $\bar x(t)-\bar
x(t-h)=\alpha_n$, for some $n$. If we temporarily let $d/dt$ denote
the backward derivative, we therefore obtain, when letting $h\rightarrow 0$
\begin{equation}\label{eq:backder}
\frac{d}{dt} u\big(\bar x(t),t) \geq p_t + p_x\cdot\alpha_n.
\end{equation}
The double sided and the backward derivative of $u\big(\bar
x(t),t\big)$ differ on a set of measure zero, so there is no problem
in using the backward derivative in \eqref{eq:upperstart}. By the
relation \eqref{eq:Ltilde} between the running cost and the Hamiltonian, the
following inequality holds:
\begin{equation}\label{eq:Hamineq}
p_x\cdot\alpha_n + \tilde L\big(\bar x(t),\alpha_n\big) \geq \tilde H\big(\bar x(t),p_x\big).
\end{equation} 
The fact that $u$ is a viscosity solution of \eqref{eq:HJ} implies that
\begin{equation}\label{eq:pvisc}
p_t + H\big(\bar x(t),p_x\big)\geq 0.
\end{equation}
When combining equations \eqref{eq:backder}, \eqref{eq:Hamineq}, and
\eqref{eq:pvisc} with the error representation \eqref{eq:upperstart}, 
we find that 
\begin{multline}\label{eq:udifffin}
\bar u(x_0,0)-u(x_0,0)\geq \sum_{n=0}\sp{N-1}\Big(\Delta t
\tilde L(x_n,\alpha_n) - \int_{t_n}\sp{t_{n+1}} \tilde L\big(\bar
x(t),\alpha_n\big)dt\Big)\\
=\sum_{n=0}\sp{N-1}
\int_{t_n}\sp{t_{n+1}}\big(\tilde L(x_n,\alpha_n)-\tilde L(\bar x(t),\alpha_n)\big)dt.
\end{multline}
With the quadratic term added to the Hamiltonian,
the supremum in the definition of $\tilde L$  in \eqref{eq:Ltilde} is
always attained. Denote by $\lambda\sp *$ a maximum point in $\lambda$
for the coordinates $(\bar x(t),\alpha_n)$, i.e.\ 
\begin{equation*}
\tilde L\big(\bar x(t),\alpha_n\big)=-\alpha_n\cdot\lambda\sp * +
\tilde H\big(\bar x(t),\lambda\sp *\big).
\end{equation*}
We thereby have the lower bound
\begin{multline}\label{eq:Ldifflower}
\tilde L(x_n,\alpha_n)-\tilde L\big(\bar x(t),\alpha_n\big) \geq
-\alpha_n\cdot\lambda\sp * + \tilde H(x_n,\lambda\sp *) -
\Big(-\alpha_n\cdot\lambda\sp * + \tilde H(\bar x(t),\lambda\sp *
\Big)\\
 = \tilde H(x_n, \lambda\sp *)-\tilde H\big(x(t),\lambda\sp*\big) =
H(x_n, \lambda\sp *)- H\big(x(t),\lambda\sp*\big).
\end{multline}
It is straightforward to show that for each $\varepsilon > 0$, there
exists a $K$, such that every maximizer $\lambda\sp *$ in the
definition of $\tilde L$, \eqref{eq:Ltilde}, satisfies
\begin{equation}\label{eq:lambdastarbound}
|\lambda\sp *|\leq (C_3+1)e\sp{C_2 T}-1 + \varepsilon.
\end{equation}
When we combine the bound on $H$ in \eqref{eq:Hamiltonianestimates} with \eqref{eq:udifffin},
\eqref{eq:Ldifflower}, and \eqref{eq:lambdastarbound}, we obtain 
\begin{equation*}
\bar u(x_0,0)-u(x_0,0) \geq -\half C_1C_2(C_3+1)e\sp{C_2 T}T\Delta t.
\end{equation*}

\end{proof}

\bibliographystyle{plain}

\bibliography{references} 


\end{document}